\documentclass[12pt, reqno]{amsart}
\usepackage{amsmath, amsthm, amscd, amsfonts, amssymb, graphicx, hyperref, color}
\usepackage[utf8]{inputenc}
\usepackage[T1]{fontenc}
\usepackage{enumitem}
\usepackage{mathtools}
\usepackage{empheq}
\newtheorem{theorem}{Theorem}[section]
\newtheorem{lemma}[theorem]{Lemma}
\newtheorem{proposition}[theorem]{Proposition}
\newtheorem{corollary}[theorem]{Corollary}
\theoremstyle{definition}

\theoremstyle{remark}
\newtheorem{remark}[theorem]{Remark}
\numberwithin{equation}{section}

\begin{document}
\setcounter{page}{1}


\title[Harnack parts for 5-by-5 truncated shift]
      {Harnack parts for 5-by-5 truncated shift with numerical radius one}
       
\author[ M. Benharrat]{Mohammed Benharrat$^{1}$}
\address{$^{1}$ Ecole Nationale Polytechnique d'Oran-Maurice Audin (Ex. ENSET d'Oran), 
	BP 1523 Oran-El M'naouar, 31000 Oran, Alg\'{e}rie.
 Laboratory of Fundamental and Applicable Mathematics of Oran (LMFAO)}
\email{\textcolor[rgb]{0.00,0.00,0.84}{mohammed.benharrat@enp-oran.dz}}



\date{\today.
}
\begin{abstract} 
	We provide a complete description of the Harnack part for normalized truncated shift of size five with numerical radius  one. We prove that any operator in this Harnack component must assume one of two distinct forms: either it belongs to the unitary orbit of the shift, thereby preserving its norm, or it is a structured nilpotent matrix with a different norm. Using polynomial methods derived from the kernel conditions, we establish that any element is necessarily nilpotent with same order of the truncated shift. These results reveal that the Harnack equivalence class exhibits a significantly richer algebraic structure in dimension five than previously observed in lower-dimensional cases.
\end{abstract}
 \maketitle
\section{Introduction}

Let $H$ be a complex Hilbert space and denote by $B(H)$ the algebra of all bounded linear operators on $H$.
In the finite-dimensional case where $\dim(H)=n$, $B(H)$ is identified with the algebra of
$n\times n$ complex matrices.

Let $\rho >0$. An operator $T \in B(H)$  is said to be  a \emph{$\rho$-contraction} 
if $T$ admits a \emph{unitary $\rho$-dilation}. 
This means that there exist a Hilbert space $\mathcal{H}$ containing $H$ as a closed subspace and a 
unitary operator $U \in B(\mathcal{H})$ such that 
\begin{equation}\label{eq:rhodil}
	T^n = \rho P_H U^n|_H, \quad n \in \mathbb{N}^{\ast},
\end{equation}
where $P_H$ is the orthogonal projection from $\mathcal{H}$ onto $H$. Let $C_\rho(H)$ denote 
the set of all $\rho$-contractions. These classes were introduced by B. Sz.-Nagy and C. Foia\c{s} 
in \cite{SzNF1} (see also \cite{SzNFBK}). The class $C_1(H)$ coincides with the class of all 
contractions, i.e., operators $T$ such that $\|T\| \leq 1$ \cite{SzN}. The class $C_2(H)$ corresponds 
precisely to the set of all operators $T \in B(H)$ whose numerical radius is less than or equal to one \cite{Be}. Recall  that for an operator $T \in B(H)$, the numerical radius is defined by
\[ w(T) = \sup\{ |\langle Tx,x\rangle | : x \in H,\ \|x\| = 1\} \]
 More generally,   the \emph{$\rho$-numerical radius} (or \emph{operator radius}) of an operator 
$T$ is defined by
\begin{equation}\label{eq:1-rad}
	w_\rho(T) := \inf\{ \gamma > 0 : \tfrac{1}{\gamma}T \in C_\rho(H)\},
\end{equation}
see \cite{H1} and \cite{W}. Thus, operators in $C_\rho(H)$ are contractions with respect to the 
\emph{$\rho$-numerical radius}, which means $T \in C_\rho(H)$ if and only if $w_\rho(T) \leq 1$. Note that $w_1(T) = \|T\|$, $w_2(T) = w(T)$, and 
$\lim_{\rho \to \infty} w_\rho(T) = r(T)$, where $r(T)$ is the \emph{spectral radius} of $T$.

A fundamental tool in the study of the classes $C_\rho(H)$ is the operator-valued $\rho$-kernel associated with any bounded operator $T$ whose spectrum lies in the closed unit disk $\overline{\mathbb{D}}$:
\begin{equation}\label{eq:kernel}
	K_{z}^{\rho}(T) = (I - \overline{z}T)^{-1} + (I - zT^{*})^{-1} + (\rho - 2)I, \quad z \in \mathbb{D}.
\end{equation}
This kernel provides a complete characterization of $\rho$-contractions: an operator $T$ belongs to $C_\rho(H)$ if and only if $\sigma(T) \subseteq \overline{\mathbb{D}}$ and $K_{z}^{\rho}(T) \geq 0$ for all $z \in \mathbb{D}$ (see \cite{CaF_2}).

For $T_0, T_1 \in C_\rho(H)$, recall that $T_1$ is \emph{Harnack dominated} by $T_0$, denoted by $T_{1} \stackrel{H}{\prec} T_{0}$ (see \cite{CaSu}), if there exists a constant $c \geq 1$ such that
\begin{equation*}
	\operatorname{Re} p(T_1) \leq c^{2} \operatorname{Re} p(T_0) + (c^{2}-1)(\rho - 1)\operatorname{Re} p(O_{H})
\end{equation*}
for any polynomial $p$ with $\operatorname{Re} p \geq 0$ on $\overline{\mathbb{D}}$, where $\operatorname{Re} z$ denotes the real part of a complex number $z$ and $O_H$ denotes the zero operator on $H$. A detailed description of Harnack domination and other equivalent definitions is given in \cite[Theorem 3.1]{CaSu}. In particular, it is proved in \cite{CaSu} that Harnack domination is equivalent to
\begin{equation}\label{harnack}
	K_{z}^{\rho}(T_{1}) \leq c^{2} K_{z}^{\rho}(T_{0}) \quad \text{for all } z \in \mathbb{D},
\end{equation}
for some constant $c \geq 1$. Thus, the operator-valued $\rho$-kernel plays a central role in the Harnack analysis of operators (see, for instance, \cite{Cassier_2}, \cite{CaSu}, \cite{CaBeBel2018}, and \cite{CassierBenharrat2020}). This is primarily due to the fact that it allows us to apply harmonic analysis methods directly within the setting of operator theory.

The relation $\stackrel{H}{\prec}$ is a preorder (reflexive and transitive) on $C_\rho(H)$ and induces an equivalence relation known as \emph{Harnack equivalence}. The associated equivalence classes are called the \emph{Harnack parts} of $C_\rho(H)$. Thus, we say that $T_{1}$ and $T_{0}$ are Harnack equivalent, denoted by $T_{1} \stackrel{H}{\sim} T_{0}$, if they belong to the same Harnack part. Classifying the equivalence classes induced by this preorder is a challenging problem and remains an active topic of research.

Foia\c{s} \cite{Foias} proved that, in the case $\rho=1$, the Harnack part of $C_1(H)$ containing the zero operator $O_H$ is exactly the class of all strict contractions ($\|T\| < 1$). This work was extended to $C_\rho(H)$ by Cassier and Suciu \cite[Theorem 4.4]{CaSu}, who proved that the equivalence class of the zero operator $O_H$ is exactly the class of all strict $\rho$-contractions (i.e., operators $T \in C_\rho(H)$ such that $w_{\rho}(T) < 1$).

An interesting question is to describe the Harnack parts of $\rho$-contractions $T$ with $\rho$-numerical radius equal to one. Few results address this question in the literature, primarily for $C_1(H)$ operators with norm one; see \cite{AnSuTi}, \cite{KSS}, and \cite{BaTiSu}. The case $\rho \neq 1$ is addressed in \cite{CaBeBel2018}, \cite{CassierBenharrat2020}, and \cite{CaNaBe2024}. In particular, results are available for the truncated shift viewed as a $2$-contraction with numerical radius one in dimensions less than four.

Let the normalized truncated shift $S$ of size $n+1$ be defined in the canonical basis of $\mathbb{C}^{n+1}$ by
\begin{equation}\label{S-truncated shift n+1}
	S = S_{n+1}(a) = \begin{bmatrix}
		0 & a & 0 & \cdots & 0 \\
		0 & 0 & \ddots & \ddots & \vdots \\
		\vdots & \ddots & \ddots & \ddots & 0 \\
		\vdots & & \ddots & 0 & a \\
		0 & \cdots & \cdots & 0 & 0
	\end{bmatrix},
\end{equation} 
where $a = \left(\cos\dfrac{\pi}{n+2}\right)^{-1}$. 
It was established in \cite[Theorem 3.7]{CassierBenharrat2020} that if $T \in C_2(\mathbb{C}^{n+1})$ satisfies $T \stackrel{H}{\sim} S$ and $\|T\| = \|S\| = \left(\cos\left(\dfrac{\pi}{n+2}\right)\right)^{-1}$, then $T=S$ if $n+1$ is an odd natural number, whereas $T$ lies in the unitary orbit of $S$ if $n+1$ is an even natural number.

It was shown that the Harnack part of a truncated shift of order two is trivial \cite[Theorem 3.3]{CaBeBel2018}, while the Harnack part of $S_3(a)$ with $a=\sqrt{2}$ corresponds to an orbit associated with the action of a group of unitary diagonal matrices,
\begin{equation}\label{S3}
	T = U_{\theta}^{*} S_3 U_\theta \quad \text{with} \quad U_\theta =
	\begin{bmatrix}
		e^{i\theta} & 0 & 0 \\
		0 & 1 & 0 \\
		0 & 0 & e^{i\theta} 
	\end{bmatrix},\qquad \theta \in \mathbb{R}.
\end{equation}
Thus, all elements of the Harnack part of $S_3$ are unitarily equivalent to $S_3$; see \cite[Theorem 3.1]{CaBeBel2018}. 

Recently, Cassier, Naimi, and Benharrat \cite{CaNaBe2024} provided a complete description of the Harnack part of $S$ for dimension $4$. More precisely, the Harnack part of $S$ in $C_2(\mathbb{C}^{4})$ consists of all matrices $T \in C_2(\mathbb{C}^4)$ of the form
\begin{equation*}
	T=S \quad \text{or} \quad T = 2\begin{bmatrix}
		0 & -\dfrac{1}{a} & 0 & 0 \\
		0 & 0 & \left(1-\dfrac{1}{a^2}\right) & 0 \\
		0 & 0 & 0 & -\dfrac{1}{a} \\
		0 & 0 & 0 & 0 \\
	\end{bmatrix},
\end{equation*}
where $a = \dfrac{1}{\cos(\pi/5)}$, see \cite[Theorem 3.3]{CaNaBe2024}.
This result highlights that the rigidity observed in Theorem 3.7 of \cite{CassierBenharrat2020}  does not necessarily persist in higher dimensions. In light of this fact, it is natural to investigate the structure of the Harnack part of the truncated shift $S$ in dimensions greater than four.

The purpose of this paper is to analyze the Harnack part of truncated shifts of size $5$ whose numerical radius is equal to one. We prove that the Harnack part of $S$ in $C_2(\mathbb{C}^{5})$ contains all matrices $T \in C_2(\mathbb{C}^5)$ of the form
\begin{equation}\label{T case123}
	T = \dfrac{2}{\sqrt{3}} \begin{bmatrix}
		0 & 1 & 0 & 0 & 0 \\
		0 & 0 & e^{i\theta} & 0 & 0 \\
		0 & 0 & 0 & e^{-i\theta} & 0 \\
		0 & 0 & 0 & 0 & 1 \\
		0 & 0 & 0 & 0 & 0
	\end{bmatrix},\quad \text{or} \quad 
	T = \begin{bmatrix}
		0 & -\sqrt{3} & 0 & 0 & 0 \\
		0 & 0 & \dfrac{1}{\sqrt{2}}e^{i\theta} & 0 & 0 \\
		0 & 0 & 0 & \dfrac{1}{\sqrt{2}}e^{-i\theta} & 0 \\
		0 & 0 & 0 & 0 & -\sqrt{3} \\
		0 & 0 & 0 & 0 & 0
	\end{bmatrix},
\end{equation} 
with $\theta \in \mathbb{R}$. 

Our main result reveals that the structure of this equivalence class is richer than previously known for lower dimensions. Specifically, we establish that any operator $T$ Harnack equivalent to $S$ must assume one of two distinct forms: either it lies in the unitary orbit of $S$ with equal norm, or it is a specific structured matrix with different norm properties. This finding not only completes the analysis of Harnack parts for this specific case but also highlights a significant distinction between the Harnack part in dimension $5$ and those in dimensions $2$, $3$, and $4$.
\section{Harnack parts for 5-by-5 truncated shift}

Let $S$ be the $w_2$-normalized truncated shift in $C_{2}(\mathbb{C}^{5})$, defined in the canonical basis of $\mathbb{C}^5$ by
\begin{equation}\label{truncated shift}
	S = a\begin{bmatrix}
		0 & 1 & 0 & 0 & 0 \\
		0 & 0 & 1 & 0 & 0 \\
		0 & 0 & 0 & 1 & 0 \\
		0 & 0 & 0 & 0 & 1 \\
		0 & 0 & 0 & 0 & 0
	\end{bmatrix},
\end{equation} 
where $a = \dfrac{1}{\cos(\pi/6)}$ (ensuring that $w(S)=1$).

In this section, we provide a complete description of the Harnack part of $S$. We show that this Harnack part exhibits structural features that differ significantly from previously known results in lower dimensions.

Suppose $T$ belongs to the Harnack part of $S$. By \cite[Theorem~2.4]{CaNaBe2024}, with respect to the orthogonal decomposition $\mathbb{C}^{5}= \mathbb{C}e_0 \oplus (\mathbb{C}e_1 \oplus\mathbb{C}e_{2}\oplus\mathbb{C}e_3)\oplus\mathbb{C}e_4$, the operator $T$ necessarily admits the block form
\begin{equation}\label{form of T}
	T = \begin{bmatrix}
		0 & a_1 & a_2 & a_3 & 0 \\
		0 & r_{11} & r_{12} & r_{13} & b_1 \\
		0 & r_{21} & r_{22} & r_{23} & b_2 \\
		0 & r_{31} & r_{32} & r_{33} & b_3 \\
		0 & 0 & 0 & 0 & 0 \\
	\end{bmatrix} 
	= \begin{bmatrix}
		0 & A & 0 \\
		0 & R & B \\
		0 & 0 & 0 \\
	\end{bmatrix},
\end{equation} 
where 
\[
A = \begin{bmatrix} a_1 & a_2 & a_3 \end{bmatrix}, \quad 
B = \begin{bmatrix} b_1 \\ b_2 \\ b_3 \end{bmatrix}, \quad \text{and} \quad 
R = \begin{bmatrix} 
	r_{11} & r_{12} & r_{13} \\
	r_{21} & r_{22} & r_{23} \\ 
	r_{31} & r_{32} & r_{33} 
\end{bmatrix}.
\]

Recall that the operator-valued $2$-kernel is defined by
\begin{equation}\label{kernel-2}
	K_{z}^{2}(T) = (I-\overline{z}T)^{-1}+(I-zT^{*})^{-1}, \quad z\in \mathbb{D}.
\end{equation}

By Proposition~3.4 of \cite{CassierBenharrat2020}, we obtain the following characterization.

\begin{corollary}\label{Corollary2S5}
	Let $T \in C_2(\mathbb{C}^{5})$. Then the following assertions are equivalent:  
	\begin{enumerate}
		\item[(i)] $T$ belongs to the Harnack part of the $w_2$-normalized truncated shift $S$ in $C_2(\mathbb{C}^{5})$.
		\item[(ii)] The following conditions hold:
		\begin{itemize}
			\item[(a)] The matrix $T$ has the block representation given in \eqref{form of T}, with $\sigma(R)\subset \mathbb{D}$;  
			\item[(b)] For all $z\in \mathbb{T}$, 
			\begin{equation}\label{KER}		
				K_z^2(T)\, v(z) = 0, \quad \text{where} \quad v(z) =
				\begin{bmatrix}
					v_0 \\
					v_1 z \\
					0 \\
					-v_1 z^3 \\
					-v_0 z^4
				\end{bmatrix};
			\end{equation}
			\item[(c)] We have
			\begin{align}
				&\det(R) = 0, \label{p1S5}\\	
				&M(z) := I_{2} - \frac{1}{4}(A^*A + BB^*) - \operatorname{Re}(\overline{z}R) \geq 0, \quad \forall z\in \mathbb{T}, \label{p2S5}
			\end{align}
			where $I_{2}$ denotes the identity matrix of size $3$.
		\end{itemize}
	\end{enumerate}
\end{corollary}

\begin{remark}
	In condition (b) above, the vector $v(z)$ reflects the specific structure of the kernel null space associated with the truncated shift. Its explicit form is dictated by the orthogonal decomposition of $\mathbb{C}^5$ used in \eqref{form of T}.
\end{remark}

The following result translates the positivity condition in Corollary~\ref{Corollary2S5} into explicit algebraic relations. It follows from Theorem~2.4 of \cite{CaNaBe2024}.

\begin{corollary}\label{Corollary2}
	Let $T \in C_2(\mathbb{C}^{5})$. Then the following assertions are equivalent:  
	\begin{enumerate}
		\item[(i)] $T$ belongs to the Harnack part of the $w_2$-normalized truncated shift $S$ in $C_2(\mathbb{C}^{5})$.
		\item[(ii)] The following conditions hold,
		\begin{itemize}
			\item[(a)] The matrix $T$ has the block representation given in \eqref{form of T}, with $\sigma(R)\subset \mathbb{D}$;  
			\item[(b)] For all $z\in \mathbb{T}$, the kernel condition \eqref{KER} holds.
			\item[(c)] The following relations are satisfied:
			\begin{align}
				&\det(R) = 0, \label{p1}\\	
				&\det\left( R^*\vec{i}, R^*\vec{j}, M\vec{k}\right) + \det\left( R^*\vec{i}, M\vec{j}, R^*\vec{k}\right) + \det\left( M\vec{i}, R^*\vec{j}, R^*\vec{k}\right) = 0, \label{p2}\\
				&\det\left( R^*\vec{i}, M\vec{j}, M\vec{k}\right) + \det\left( M\vec{i}, M\vec{j}, R^*\vec{k}\right) + \frac{1}{4}\det\left( R\vec{i}, R^*\vec{j}, R^*\vec{k}\right) \nonumber\\
				&\quad + \frac{1}{4}\det\left( R^*\vec{i}, R\vec{j}, R^*\vec{k}\right) + \frac{1}{4}\det\left( R^*\vec{i}, R^*\vec{j}, R\vec{k}\right) = 0, \label{p3}\\
				&\det(M) + \frac{1}{4}\det\left( R^*\vec{i}, R\vec{j}, M\vec{k}\right) + \frac{1}{4}\det\left( R\vec{i}, R^*\vec{j}, M\vec{k}\right) \nonumber\\
				&\quad + \frac{1}{4}\det\left( M\vec{i}, R^*\vec{j}, R\vec{k}\right) + \frac{1}{4}\det\left(M \vec{i}, R\vec{j}, R^*\vec{k}\right) \nonumber\\
				&\quad + \frac{1}{4}\det\left( R^*\vec{i}, M\vec{j}, R\vec{k}\right) + \frac{1}{4}\det\left( R\vec{i}, M\vec{j}, R^*\vec{k}\right) = 0, \label{p4}\\
				&3 - \frac{1}{4}\left( \| A \|^2 + \| B \|^2 \right) > | \operatorname{Tr}(R)|, \label{p5}
			\end{align}
			where $\vec{i} = (1, 0, 0)^\top$, $\vec{j} = (0, 1, 0)^\top$, $\vec{k} = (0, 0, 1)^\top$, and $M := I_{2} - \frac{1}{4}(A^*A + BB^*)$.
		\end{itemize}
	\end{enumerate}
\end{corollary}
We begin by considering the case where $0$ is an eigenvalue of the block $R$ with multiplicity at least $2$.
\begin{proposition}\label{Harnack_part_S_5} 
	Assume that the spectrum of $R$ is $\sigma(R)=\{ 0, \lambda \}$, with $\lambda \in \mathbb{D}$, and that the eigenvalue $0$ has algebraic multiplicity at least $2$. If $T\in C_2(\mathbb{C}^5)$ lies in the Harnack part of $S$, then $0$ must have multiplicity $3$ (i.e., $\lambda=0$), and $T$ takes one of the following forms,
	\begin{equation}\label{T case12}
		T=\frac{2}{\sqrt{3}} \begin{bmatrix}
			0 & 1 & 0 & 0 & 0 \\
			0 & 0 & e^{i\theta} & 0 & 0 \\
			0 & 0 & 0 & e^{-i\theta} & 0 \\
			0 & 0 & 0 & 0 & 1 \\
			0 & 0 & 0 & 0 & 0
		\end{bmatrix},\quad \text{or} \quad 
		T=  \begin{bmatrix}
			0 & -\sqrt{3} & 0 & 0 & 0 \\
			0 & 0 & \frac{1}{\sqrt{2}}e^{i\theta} & 0 & 0 \\
			0 & 0 & 0 & \frac{1}{\sqrt{2}}e^{-i\theta} & 0 \\
			0 & 0 & 0 & 0 & -\sqrt{3} \\
			0 & 0 & 0 & 0 & 0
		\end{bmatrix},
	\end{equation} 
	with $\theta \in \mathbb{R}$.	
\end{proposition}

\begin{proof}
	Let $T \in C_2(\mathbb{C}^5)$ belong to the Harnack part of $S$. By Corollary~\ref{Corollary2}, $T$ admits the block form \eqref{form of T} with $\sigma(R)\subset \mathbb{D}$ and $\det(R)=0$. Proposition~2.3 of \cite{CaNaBe2024} yields $\sigma(T)= \sigma(R) \cup \{ 0 \}$; hence, by \cite[Corollary~2.2]{CaBeBel2018}, we must have $\sigma(R)=\{ 0, \lambda \}$ for some $\lambda \in \mathbb{D}$. 
	
	Since $0$ has multiplicity at least $2$, the minimal polynomial of $R$ divides $t^2(t-\lambda)$, which implies $R^3=\lambda R^2$. Consequently, the resolvent expands as
	\[
	\mathbf{R}(\overline{z}, \lambda) = (I_2-\overline{z}R)^{-1} = I_2 + \overline{z} R + \frac{\overline{z}^2}{1-\lambda \overline{z}}R^2 \qquad \text{for all } z\in \overline{\mathbb{D}},
	\]
	and $K_z^2(R) = \mathbf{R}^*(z, \overline{\lambda}) + \mathbf{R}(\overline{z}, \lambda)$.
	
	Thus, for every $z\in \mathbb{T}$, the $2$-kernel of $T$ is given by
	\[
	K_z^2(T) = \begin{bmatrix}
		2 & \overline{z} A \mathbf{R}(\overline{z}, \lambda) & \overline{z}^2 A \mathbf{R}(\overline{z}, \lambda)B\\
		z \mathbf{R}^*(z, \overline{\lambda})A^* & K_z^2(R) & \overline{z}\mathbf{R}(\overline{z}, \lambda) B\\
		z^2 B^* \mathbf{R}^*(z, \overline{\lambda})A^* & z B^* \mathbf{R}^*(z, \overline{\lambda}) & 2
	\end{bmatrix}.
	\] 
	Invoking condition \eqref{KER}, we have for all $z\in \mathbb{T}$,
	\begin{equation}\label{system1} 
		K_{z}^{2}(T) \begin{bmatrix} v_0 \\ z w(z) \\ -v_0 z^4 \end{bmatrix} = 0,
	\end{equation}
	where $w(z) = v_1(1, 0, -z^2)^\top$ and $v_k \neq 0$ for $k=0,1$. Equation \eqref{system1} is equivalent to the system
	\begin{equation}\label{system01} 
		\begin{cases}
			2v_0 + A \mathbf{R}(\overline{z}, \lambda) w(z) - z^2 v_0 A \mathbf{R}(\overline{z}, \lambda) B = 0,\\
			v_0 \mathbf{R}^*(z, \overline{\lambda})A^* + K_z^2(R)w(z) - v_0 z^2 \mathbf{R}(\overline{z}, \lambda) B = 0,\\
			v_0 B^*\mathbf{R}^*(z, \overline{\lambda})A^* + B^*\mathbf{R}^*(z, \overline{\lambda})w(z) - 2v_0 z^2 = 0.
		\end{cases} 
	\end{equation}
	The third equation in \eqref{system01} can be rewritten as
	\[
	v_0 z^2 A \mathbf{R}(z, \overline{\lambda}) B + z^2 w^*(z) \mathbf{R}(z, \overline{\lambda}) B - 2v_0 = 0.
	\]
	Adding this to the first equation of \eqref{system01} yields
	\begin{equation}\label{eqfour}
		A \mathbf{R}(\overline{z}, \lambda) w(z) = -z^2 w^*(z) \mathbf{R}(z, \overline{\lambda}) B,
	\end{equation}
	for all $z\in \mathbb{T}$. Hence, \eqref{system01} and \eqref{eqfour} are equivalent to the following polynomial identities:
	\begin{align}
		2v_0 z(z-\lambda) &+ v_1 A \Bigl(  z(z-\lambda)I_2 + (z-\lambda)R + R^2 \Bigr) 
		\begin{bmatrix} 1\\0\\-z^2 \end{bmatrix} \nonumber \\
		&\qquad  - v_0 z^2 A \Bigl( z(z-\lambda)I_2 + (z-\lambda)R + R^2 \Bigr) B = 0, \label{systm:eq1}	
	\end{align}
	\begin{align}
		v_0 &z(z-\lambda)  \Bigl( (1-\overline{\lambda}z)I_2 + z(1-\overline{\lambda}z)R^* + z^2 R^{*2} \Bigr) A^* \nonumber \\ 
		& + v_1 z(z-\lambda) \Bigl( (1-\overline{\lambda}z)I_2 + z(1-\overline{\lambda}z)R^* + z^2 R^{*2} \Bigr) \begin{bmatrix} 1\\0\\-z^2 \end{bmatrix} \nonumber \\
		& + v_1 (1-\overline{\lambda}z) \Bigl( z(z-\lambda)I_2 + (z-\lambda)R + R^{2} \Bigr) \begin{bmatrix} 1\\0\\-z^2 \end{bmatrix} \label{systm:eq2} \\
		& \qquad \qquad - v_0 z^2 (1-\overline{\lambda}z) \Bigl( z(z-\lambda)I_2 + (z-\lambda)R + R^2 \Bigr)B = 0,\nonumber 	
	\end{align}
	\begin{align}
		v_0& B^*  \Bigl( (1-\overline{\lambda}z)I_2 + z(1-\overline{\lambda}z)R^* + z^2 R^{*2} \Bigr) A^* \nonumber \\
		& + v_1 B^* \Bigl( (1-\overline{\lambda}z)I_2 + z(1-\overline{\lambda}z)R^* + z^2 R^{*2} \Bigr) \begin{bmatrix} 1\\0\\-z^2 \end{bmatrix} - 2v_0 z^2 (1-\overline{\lambda}z) = 0, \label{systm:eq3}	
	\end{align}
	and
	\begin{align}
		& A \left( z(z-\lambda)I_2 + (z-\lambda)R + R^2 \right) \begin{bmatrix} 1\\0\\-z^2 \end{bmatrix} \nonumber \\ 
		& \quad  \qquad = \begin{bmatrix} -z^2 & 0 & 1 \end{bmatrix} \left( z(z-\lambda)I_2 + (z-\lambda)R + R^2 \right) B. \label{systm:eq4} 
	\end{align}
	By analytic continuation, equations \eqref{systm:eq1}--\eqref{systm:eq4} hold for all $z\in \mathbb{C}$. In the sequel, we denote by $\mathbf{e}_1$, $\mathbf{e}_2$, and $\mathbf{e}_3$ the vectors of the canonical basis of $\mathbb{C}^3$. Comparing the coefficients of $z^k$ for $k=0,1,2,3,4$ in \eqref{systm:eq4} yields, respectively,
	\begin{equation}\label{eq4:1}
		A(-\lambda R + R^2) \mathbf{e}_1 = \mathbf{e}_3^\top (-\lambda R + R^2) B, 
	\end{equation} 
	\begin{equation}\label{eq4:2}
		A(-\lambda I_2 + R) \mathbf{e}_1 = \mathbf{e}_3^\top (-\lambda I_2 + R) B, 
	\end{equation} 
	\begin{align}
		& A \mathbf{e}_1 + A(-\lambda R + R^2) (-\mathbf{e}_3) \nonumber\\ 
		& \quad - \mathbf{e}_3^\top B + \mathbf{e}_1^\top (-\lambda R + R^2) B = 0, \label{eq4:3}
	\end{align}
	\begin{equation}\label{eq4:4} 
		A(-\lambda I_2 + R) \mathbf{e}_3 = \mathbf{e}_1^\top (-\lambda I_2 + R) B, 
	\end{equation}
	and
	\begin{equation}\label{eq4:5}  
		A \mathbf{e}_3 = \mathbf{e}_1^\top B. 
	\end{equation} 
	From \eqref{eq4:5}, we immediately obtain
	\begin{equation}\label{eq1}
		a_3 = b_1.	
	\end{equation} 
	Combining \eqref{eq4:5} and \eqref{eq4:4} gives
	\begin{equation}\label{eq4:7}
		A R \mathbf{e}_3 = \mathbf{e}_1^\top R B. 
	\end{equation}
	Thus, \eqref{eq4:3} simplifies to
	\begin{equation}
		A \mathbf{e}_1 - A R^2 \mathbf{e}_3 - \mathbf{e}_3^\top B + \mathbf{e}_1^\top R^2 B = 0, \label{eq4:6}
	\end{equation}  
	which is equivalent to
	\begin{equation}\label{eq4:6bis}
		a_1 - b_3 = A R^2 \mathbf{e}_3 - \mathbf{e}_1^\top R^2 B.
	\end{equation}
	Next, extracting the coefficient of $z$ in \eqref{systm:eq2} yields
	\[
	-\lambda v_0 A^* + v_1 (-2\lambda I_2 + R) \mathbf{e}_1 = 0,
	\]
	and consequently
	\begin{equation}\label{eq2:z}
		v_1 R \mathbf{e}_1 = \lambda v_0 A^* + 2\lambda v_1 \mathbf{e}_1. 
	\end{equation}
	Similarly, the coefficient of $z^5$ in \eqref{systm:eq2} leads to
	\[
	v_1 (-2\overline{\lambda} I_2 + R^*) \mathbf{e}_3-v_0 \overline{\lambda} B = 0,
	\]
	whence
	\begin{equation}\label{eq2:z5}
		v_1 R^* \mathbf{e}_3 = 2\overline{\lambda} v_1 \mathbf{e}_3 + v_0 \overline{\lambda} B. 
	\end{equation} 
	Taking the coefficients of $z^2$ and $z^4$ in \eqref{systm:eq2}, we obtain
	\begin{align}
		v_0 \Bigl( (1+|\lambda|^2)I_2 - \lambda R^* \Bigr)  A^* & + v_1 \Bigl( 2(1+|\lambda|^2)I_2 - \lambda R^* - \overline{\lambda}R \Bigr)  \mathbf{e}_1 \nonumber \\ 
		& - v_1 \left(  -\lambda R + R^2 \right)  \mathbf{e}_3 - v_0 \left(  -\lambda R + R^2 \right)    B = 0, \label{eq2:z2}
	\end{align}
	and
	\begin{align}
		v_0 \left(  -\overline{\lambda} R^* + R^{*2} \right)   A^*  + v_1 \bigl( -\overline{\lambda} R^* + R^{*2} \bigr)  \mathbf{e}_1& - v_1  \Bigl( 2(1+|\lambda|^2)I_2 - \lambda R^* - \overline{\lambda}R \Bigr)  \mathbf{e}_3 \nonumber \\
		& - v_0  \Bigl((1+|\lambda|^2)I_2 - \overline{\lambda} R \Bigr)  B = 0. \label{eq2:z4}
	\end{align}
	Taking the inner product of \eqref{eq2:z2} with $\mathbf{e}_1$ gives
	\begin{align}
		v_0 & \left\langle \Bigl( (1+|\lambda|^2)I_2 - \lambda R^* \Bigr)  A^*, \mathbf{e}_1 \right\rangle + v_1 \left\langle \Bigl( 2(1+|\lambda|^2)I_2 - \lambda R^* - \overline{\lambda}R \Bigr)  \mathbf{e}_1, \mathbf{e}_1 \right\rangle \nonumber \\ 
		&\qquad \qquad \qquad \qquad - v_1 \left\langle\left(  -\lambda R + R^2 \right) \mathbf{e}_3, \mathbf{e}_1 \right\rangle - v_0 \left\langle \left(  -\lambda R + R^2 \right) B, \mathbf{e}_1 \right\rangle = 0. \label{eq2:z2x100}
	\end{align}
	Using \eqref{eq4:2}, this simplifies to
	\begin{align}
		v_0 \overline{a}_1  - \lambda &\left\langle \left(  -\overline{\lambda}I_2 + R^* \right)  A^*, \mathbf{e}_1 \right\rangle + 2v_1(1+|\lambda|^2) + v_1 \left\langle \left(  -\lambda R^* - \overline{\lambda}R \right)  \mathbf{e}_1, \mathbf{e}_1 \right\rangle \nonumber \\ 
		&  \qquad \qquad -v_1 \left\langle \left(  -\lambda R + R^2 \right)  \mathbf{e}_3, \mathbf{e}_1 \right\rangle - v_0 \left\langle \left(  -\lambda R + R^2 \right)  B, \mathbf{e}_1 \right\rangle = 0. \label{eq2:z2x1002}
	\end{align}
	From \eqref{eq2:z}, we compute
	\begin{align}
		v_1 \left\langle \left(  -\lambda R^* - \overline{\lambda}R \right)  \mathbf{e}_1, \mathbf{e}_1 \right\rangle 
		&= -\lambda \left\langle \mathbf{e}_1, v_1 R \mathbf{e}_1 \right\rangle - \overline{\lambda} \left\langle v_1 R \mathbf{e}_1, \mathbf{e}_1 \right\rangle \nonumber\\
		&= -|\lambda|^2 v_0 a_1 - |\lambda|^2 v_0 \overline{a}_1 - 4|\lambda|^2 v_1. \label{eq2:R}
	\end{align}
	Taking the adjoint of \eqref{eq2:z4} yields
	\begin{align}
		v_0 A \left(  -\overline{\lambda} R + R^2 \right)   + v_1 \mathbf{e}_1^\top \left(  -\lambda R + R^{2} \right) & - v_1 \mathbf{e}_3^\top \Bigl( 2(1+|\lambda|^2)I_2 - \lambda R^* - \overline{\lambda}R \Bigr)  \nonumber \\
		& - v_0 B^* \Bigl( (1+|\lambda|^2)I_2 - \lambda R^* \Bigr)  = 0. \label{eq2:z4*}
	\end{align}
	Applying this to $\mathbf{e}_3$ gives
	\begin{align}
		v_0 A \left(-\overline{\lambda} R + R^2 \right)  \mathbf{e}_3 & + v_1 \mathbf{e}_1^\top \left(-\lambda R + R^{2} \right) \mathbf{e}_3  - 2v_1 (1+|\lambda|^2)  \nonumber\\
		& + v_1 \mathbf{e}_3^\top \left(  -\lambda R^* - \overline{\lambda}R\right)  \mathbf{e}_3- v_0 \overline{b}_3 - \lambda B^*\left(-\overline{\lambda}I_2 + R^* \right)  \mathbf{e}_3 = 0. \label{eq2:z4*0012}
	\end{align}
	Using \eqref{eq2:z5}, we find
	\begin{align}
		v_1 \mathbf{e}_3^\top \left(-\lambda R^* - \overline{\lambda}R\right)  \mathbf{e}_3 
		&= \lambda \left\langle v_1 R^* \mathbf{e}_3, \mathbf{e}_3 \right\rangle + \overline{\lambda} \left\langle \mathbf{e}_3, v_1 R^* \mathbf{e}_3 \right\rangle \nonumber\\
		&= |\lambda|^2 v_0 b_3 + |\lambda|^2 v_0 \overline{b}_3 + 4|\lambda|^2 v_1. \label{eq2:R*}
	\end{align}
	Adding \eqref{eq2:z2x1002} to \eqref{eq2:z4*0012}, and using \eqref{eq4:2}, \eqref{eq4:6bis}, \eqref{eq2:R}, and \eqref{eq2:R*}, we obtain
	\begin{equation}\label{key_real}
		(1-|\lambda|^2)\bigl(\operatorname{Re}(a_1) - \operatorname{Re}(b_3)\bigr) = 0.	
	\end{equation}
	Since $|\lambda| < 1$, it follows that
	\begin{equation}\label{key_equal}
		\operatorname{Re}(a_1) = \operatorname{Re}(b_3).	
	\end{equation}
	On the other hand, we have 
	\begin{equation}\label{eq_tr}
		v_1 \lambda = v_1 \operatorname{Tr}(R)
		= \left\langle v_1 R \mathbf{e}_1, \mathbf{e}_1 \right\rangle + \left\langle \mathbf{e}_3, v_1 R^* \mathbf{e}_3 \right\rangle+v_1\left\langle R  \mathbf{e}_2, \mathbf{e}_2 \right\rangle. 
	\end{equation}
	Substituting \eqref{eq2:z} and \eqref{eq2:z5} into \eqref{eq_tr} yields
	\begin{align*}
		v_1 \lambda 
		&= \left\langle \lambda v_0 A^* + 2\lambda v_1 \mathbf{e}_1, \mathbf{e}_1 \right\rangle + \left\langle \mathbf{e}_3, 2\overline{\lambda} v_1 \mathbf{e}_3 + v_0 \overline{\lambda} B \right\rangle +v_1r_{22}\\
		&= 2\lambda v_1 + v_0 \lambda \overline{a}_1 + 2\lambda v_1 + v_0 \lambda \overline{b}_3+v_1r_{22}.
	\end{align*}
	We deduce that 
	\begin{equation}\label{E1}
		\lambda \left( 3v_1 + v_0(\overline{a}_1+ \overline{b}_3) \right)+v_1r_{22} = 0.
	\end{equation}
On the other hand,  \eqref{eq2:z}  gives
	\begin{align*}
		0 &= -\lambda v_0 A A^* + v_1 A (R - \lambda I_2) \mathbf{e}_1 - \lambda v_1 A \mathbf{e}_1 \\
		&= -\lambda v_0 A A^* + 2\lambda v_0 - \lambda v_1 a_1. 
	\end{align*}
	Similarly,  \eqref{eq2:z5} yields 
	\begin{align*}
		0 &= v_1\mathbf{e}_3^\top  (R - 2\lambda I_2)B + v_0 \lambda B^* B \\
		&= v_1\mathbf{e}_3^\top  (R - \lambda I_2)B+v_1\mathbf{e}_3^\top B + v_0 \lambda B^* B \\
		&= -2\lambda v_0 + v_1 \lambda b_3 + v_0 \lambda B^* B.
	\end{align*} 
	Assume now that $\lambda \neq 0$. Then $v_0 |a_1|^2 + v_1 a_1 + v_0 |a_2|^2 + v_0 |a_3|^2 - 2v_0 = 0$, which shows that $a_1$ must be a real root of this quadratic equation. Simultaneously, we have $\lambda (v_0 B^* B + v_1 b_3 - 2v_0) = 0$, so $b_3$ is also real. Hence, by \eqref{key_equal}, $a_1$ and $b_3$ are two real solutions of
	\[
	v_0 x^2 + v_1 x + v_0 (|a_2|^2 + |a_3|^2 - 2) = 0.
	\]
Thus , $a_1 + b_3 = -v_1/v_0$. Substituting this into \eqref{E1} gives $r_{22} =-2\lambda$. From \eqref{eq2:z} we obtain $r_{11}=0$, and from \eqref{eq2:z5}, $r_{33} = 0$. Thus $\lambda =r_{22} =-2\lambda$, contradicting $\lambda\neq 0$. Therefore, we must have $\lambda = 0$.
	
	With $\lambda=0$, equations \eqref{eq2:z} and \eqref{eq2:z5} imply $r_{11} = r_{21} = r_{31} = 0$ and $r_{31} = r_{32} = r_{33} = 0$. Moreover, \eqref{eq_tr} yields $r_{22} = 0$.
	Extracting the coefficient of $z^2$ in \eqref{systm:eq2} gives $a_2 = a_3 = 0$ and
	\begin{equation}\label{eq:rr}
		v_0 \overline{a}_1 + 2v_1 - v_1 r_{12} r_{23} - v_0 r_{12} r_{23} b_3 = 0.
	\end{equation}
	Similarly, the coefficient of $z^4$ in \eqref{systm:eq2} yields $b_1 = b_2 = 0$ and
	\begin{equation}\label{eq:rr1}
		v_0 r_{12} r_{23} a_1 - 2v_1 + v_1 r_{12} r_{23} - v_0 \overline{b}_3 = 0.
	\end{equation}
	The coefficient of $z^3$ in \eqref{systm:eq2} leads to
	\begin{equation} 
		\begin{cases}
			r_{13}(v_1 + v_0 b_3) = 0,\\
			\overline{r_{12}}(v_0 \overline{a}_1 + v_1) = r_{23}(v_0 b_3 + v_1),\\
			r_{13}(v_1 + v_0 a_1) = 0.
		\end{cases} 
	\end{equation}
	If $v_1 + v_0 b_3 = 0$, then $a_1 = b_3 = -v_1/v_0$, and \eqref{eq:rr} forces $v_1 = 0$, a contradiction. Hence $r_{13} = 0$ and $\overline{r_{12}} = r_{23}$. Substituting into \eqref{eq:rr} and \eqref{eq:rr1} shows that $a_1$ and $b_3$ are real with $a_1 = b_3$.
	
	Thus, $R$ must be of the form
	\[
	R = \begin{bmatrix} 
		0 & r & 0 \\
		0 & 0 & \overline{r} \\ 
		0 & 0 & 0
	\end{bmatrix}.
	\]
	Consequently,
	\begin{equation}\label{truncated shift_T}
		T = \begin{bmatrix}
			0 & x & 0 & 0 & 0\\
			0 & 0 & r & 0 & 0\\
			0 & 0 & 0 & \overline{r} & 0 \\
			0 & 0 & 0 & 0 & x\\
			0 & 0 & 0 & 0 & 0
		\end{bmatrix}.
	\end{equation} 
	From the $z^2$-coefficient in \eqref{systm:eq1}, we have 
	\[
	2v_0 + v_1 x - v_1 |r|^2 x - v_0 |r|^2 x^2 = 0.
	\]   
	Using \eqref{eq:rr}, this reduces to
	\[
	2v_0 - v_0 x^2 - v_1 x = 0.
	\] 
	Since $T=S$ corresponds to $x=a$, we know $x=a$ is a root of $v_0 x^2 + v_1 x - 2v_0 = 0$. The second root is $x=-2/a$, and the sum of the roots gives $v_1/v_0 = 2/a - a$.
	
	From \eqref{eq:rr}, we further deduce $|r|^2 = 2 - x^2/2$. We distinguish two cases:
	
	\textbf{Case 1.} If $x = a$, then $|r|^2 = 2 - a^2/2 = a^2$, so $r = a e^{i\theta}$. Thus 
	\begin{equation}\label{T case1}
		T = a \begin{bmatrix}
			0 & 1 & 0 & 0 & 0\\
			0 & 0 & e^{i\theta} & 0 & 0\\
			0 & 0 & 0 & e^{-i\theta} & 0 \\
			0 & 0 & 0 & 0 & 1\\
			0 & 0 & 0 & 0 & 0
		\end{bmatrix}.
	\end{equation} 
	
	\textbf{Case 2.} If $x = -2/a = -\sqrt{3}$, then $|r|^2 = 2(1 - 1/a^2) = 1/2$, so $r = \frac{1}{\sqrt{2}} e^{i\theta}$. Hence
	\begin{equation}\label{T case2}
		T = \begin{bmatrix}
			0 & -\sqrt{3} & 0 & 0 & 0\\
			0 & 0 & \frac{1}{\sqrt{2}} e^{i\theta} & 0 & 0\\
			0 & 0 & 0 & \frac{1}{\sqrt{2}} e^{-i\theta} & 0 \\
			0 & 0 & 0 & 0 & -\sqrt{3} \\
			0 & 0 & 0 & 0 & 0
		\end{bmatrix}.
	\end{equation} 
	
	It remains to verify that these matrices indeed lie in the Harnack part of $S$ in $C_2(\mathbb{C}^5)$. By Corollary~\ref{Corollary2}, it suffices to check conditions (a)--(c). Condition (a) holds trivially. Condition (b) is satisfied by construction, as our derivation shows that the system $K_z^2(T)v(z)=0$ reduces to identities \eqref{systm:eq1}--\eqref{systm:eq4}, which are fulfilled for the obtained parameters. For condition (c), consider Case 2 (Case 1 is analogous and simpler). Here
	\[
	A = \begin{bmatrix} -\sqrt{3} & 0 & 0 \end{bmatrix}, \quad
	B = \begin{bmatrix} 0 \\ 0 \\ -\sqrt{3} \end{bmatrix}, \quad
	R = \frac{1}{\sqrt{2}} \begin{bmatrix}
		0 & e^{i\theta} & 0 \\ 
		0 & 0 & e^{-i\theta} \\
		0 & 0 & 0
	\end{bmatrix},
	\] 
	and 
	\[
	M = I_2 - \frac{1}{4}(A^*A + BB^*) = \begin{bmatrix}
		\frac{1}{4} & 0 & 0 \\ 
		0 & 1 & 0 \\
		0 & 0 & \frac{1}{4}
	\end{bmatrix}.
	\] 
	Direct computation verifies \eqref{p1}--\eqref{p5}. In particular,
	\[
	3 - \frac{1}{4}\bigl( \| A \|^2 + \| B \|^2 \bigr) = \frac{3}{2} > | \operatorname{Tr}(R) | = 0.
	\]
	Thus, all conditions of Corollary~\ref{Corollary2} are verified, confirming $T \stackrel{H}{\sim} S$. This completes the proof.
\end{proof}
We now complete the spectral analysis of the block $R$ by considering the case where $0$ is an eigenvalue of multiplicity $1$. Specifically, we assume $\sigma(R)=\{0, \lambda_1, \lambda_2\}$ with $\lambda_1, \lambda_2 \in \mathbb{D} \setminus \{0\}$ and $\lambda_1 \neq \lambda_2$. The following proposition shows that this configuration is incompatible with the Harnack equivalence conditions.

\begin{proposition}\label{Prop:case2}
	There exists no operator $T$ in the Harnack part of $S$ in $C_2(\mathbb{C}^{5})$ such that the block $R$ in the decomposition \eqref{form of T} satisfies $\sigma(R)=\{0, \lambda_1, \lambda_2\}$ with $\lambda_1, \lambda_2 \in \mathbb{D} \setminus \{0\}$ and $\lambda_1 \neq \lambda_2$.
\end{proposition}

The following lemma provides the explicit resolvent expansion required for the proof.

\begin{lemma}\label{lemma:Rn_unified}
	Let $R \in B(\mathbb{C}^3)$ such that $\sigma(R)=\{0, \lambda_1, \lambda_2\}$ with $\lambda_1, \lambda_2 \in \mathbb{D} \setminus \{0\}$. Then for any integer $n \geq 1$, the power $R^n$ can be expressed as
	\begin{equation}\label{eq:Rn_unified}
		R^n = A_n R^2 + B_n R,
	\end{equation}
	where the coefficients $A_n, B_n$ are given by
	\begin{equation}\label{eq:AnBn_unified}
		A_n = \begin{cases} 
			\displaystyle \frac{\lambda_1^{n-1} - \lambda_2^{n-1}}{\lambda_1 - \lambda_2}, & \text{if } \lambda_1 \neq \lambda_2, \\[10pt]
			(n-1)\lambda_1^{n-2}, & \text{if } \lambda_1 = \lambda_2,
		\end{cases}
		\quad 
		B_n = \begin{cases} 
			\displaystyle -\lambda_1 \lambda_2 \frac{\lambda_1^{n-2} - \lambda_2^{n-2}}{\lambda_1 - \lambda_2}, & \text{if } \lambda_1 \neq \lambda_2, \\[10pt]
			-(n-2)\lambda_1^{n-1}, & \text{if } \lambda_1 = \lambda_2.
		\end{cases}
	\end{equation}
	Consequently, for any $z \in \overline{\mathbb{D}}$, the resolvent expansion $(I-\overline{z}R)^{-1} = I_2 + c_1(z)R + c_2(z)R^2$ holds with coefficients
	\begin{align}
		c_1(z) &= \frac{\overline{z} - \overline{z}^2(\lambda_1 + \lambda_2)}{(1-\overline{z}\lambda_1)(1-\overline{z}\lambda_2)}, \label{c1} \\
		c_2(z) &= \frac{\overline{z}^2}{(1-\overline{z}\lambda_1)(1-\overline{z}\lambda_2)}. \label{c2}
	\end{align}
\end{lemma}

\begin{proof}
	The characteristic polynomial of $R$ is $\chi_R(t) = t(t-\lambda_1)(t-\lambda_2)$. By the Cayley--Hamilton theorem, $R$ satisfies this polynomial, implying $R^3 = (\lambda_1+\lambda_2)R^2 - \lambda_1\lambda_2 R$. Thus, $R^n \in \operatorname{span}\{R, R^2\}$ for all $n \geq 1$.
	
	To determine $A_n$ and $B_n$, we seek a polynomial $p(t) = A_n t^2 + B_n t$ such that $p(R) = R^n$. This requires $p(t)$ to interpolate $f(t)=t^n$ on the spectrum of $R$.
	\begin{itemize}
		\item \textbf{Case $\lambda_1 \neq \lambda_2$:} The spectrum consists of distinct points $\{0, \lambda_1, \lambda_2\}$. The interpolation conditions $p(\lambda_i) = \lambda_i^n$ for $i=1,2$ yield a linear system whose solution gives the quotient formulas in \eqref{eq:AnBn_unified}.
		\item \textbf{Case $\lambda_1 = \lambda_2 = \lambda$:} The eigenvalue $\lambda$ has multiplicity $2$. We must satisfy the Hermite interpolation conditions $p(\lambda) = \lambda^n$ and $p'(\lambda) = n\lambda^{n-1}$. Solving the resulting system yields the derivative-based formulas in \eqref{eq:AnBn_unified}. Note that these expressions are precisely the limits of the distinct-case formulas as $\lambda_2 \to \lambda_1$.
	\end{itemize}
	
	For the resolvent, we use the Neumann series $(I-\overline{z}R)^{-1} = \sum_{n=0}^\infty \overline{z}^n R^n$, which converges for $|z| \leq 1$. Substituting \eqref{eq:Rn_unified} gives
	\[
	(I-\overline{z}R)^{-1} = I_2 + \left(\sum_{n=1}^\infty A_n \overline{z}^n\right) R^2 + \left(\sum_{n=1}^\infty B_n \overline{z}^n\right) R.
	\]
	Identifying $c_2(z) = \sum_{n=1}^\infty A_n \overline{z}^n$ and $c_1(z) = \sum_{n=1}^\infty B_n \overline{z}^n$:
	\begin{itemize}
		\item If $\lambda_1 \neq \lambda_2$, summing the geometric progressions yields \eqref{c1} and \eqref{c2} directly.
		\item If $\lambda_1 = \lambda_2 = \lambda$, the series become arithmetic-geometric. For example, $c_2(z) = \sum_{n=1}^\infty (n-1)\lambda^{n-2}\overline{z}^n = \overline{z}^2 \sum_{k=0}^\infty (k+1)(\lambda\overline{z})^k = \frac{\overline{z}^2}{(1-\overline{z}\lambda)^2}$.
	\end{itemize}
	In both cases, the resulting expressions coincide with the unified formulas \eqref{c1} and \eqref{c2}, where the denominator $(1-\overline{z}\lambda_1)(1-\overline{z}\lambda_2)$ naturally reduces to $(1-\overline{z}\lambda)^2$ in the repeated case.
\end{proof}

\begin{proof}[Proof of Proposition \ref{Prop:case2}]
	Suppose, for contradiction, that there exists $T$ in $C_2(\mathbb{C}^{5})$ such that $T \stackrel{H}{\sim} S$ and $\sigma(R)=\{0, \lambda_1, \lambda_2\}$ with $\lambda_1, \lambda_2 \in \mathbb{D} \setminus \{0\}$. By Lemma~\ref{lemma:Rn_unified}, the resolvent $(I_2 - \overline{z}R)^{-1}$ expands as a quadratic polynomial in $R$. For $z \in \mathbb{D}$, we have
	\begin{equation}\label{resolvent_expansion}
		\mathbf{R}(\overline{z}, \lambda_1, \lambda_2) = (I_2 - \overline{z}R)^{-1} =  I_2 + c_1(z) R + c_2(z) R^2,
	\end{equation}
	where $c_1(z)$ and $c_2(z)$ are given by \eqref{c1} and \eqref{c2}. Note that these coefficients have poles at $z = 1/\overline{\lambda}_1$ and $z = 1/\overline{\lambda}_2$, which lie outside $\overline{\mathbb{D}}$ since $|\lambda_1|, |\lambda_2| < 1$.
	
	As in the proof of Proposition~\ref{Harnack_part_S_5}, the kernel condition $K_z^2(T)v(z) = 0$ yields the system
	\begin{equation}\label{system01_2} 
		\begin{cases}
			2v_0 + A \mathbf{R}(\overline{z}, \lambda_1, \lambda_2) w(z) - z^2 v_0 A \mathbf{R}(\overline{z}, \lambda_1, \lambda_2) B = 0,\\
			v_0 \mathbf{R}^*(z, \overline{\lambda}_1, \overline{\lambda}_2) A^* + K_z^2(R) w(z) - v_0 z^2 \mathbf{R}(\overline{z}, \lambda_1, \lambda_2) B = 0,\\
			v_0 B^* \mathbf{R}^*(z, \overline{\lambda}_1, \overline{\lambda}_2) A^* + B^* \mathbf{R}^*(z, \overline{\lambda}_1, \overline{\lambda}_2) w(z) - 2v_0 z^2 = 0.
		\end{cases} 
	\end{equation}
	Combining the first and third equations eliminates the constant term $2v_0$, yielding the analogue of \eqref{eqfour},
	\begin{equation}\label{eqfour_2}
		A \mathbf{R}(\overline{z}, \lambda_1, \lambda_2) w(z) = -z^2 w^*(z) \mathbf{R}(z, \overline{\lambda}_1, \overline{\lambda}_2) B.
	\end{equation}
	Let $P_1(\zeta) = (1 - \zeta\lambda_1)(1 - \zeta\lambda_2) = 1 - \sigma_1\zeta + \sigma_2\zeta^2$. The key equation \eqref{eqfour_2} can be rewritten as,
	\begin{align}\label{poly_identity}
		&A \left[ P_1(\overline{z}) I_2 + (\overline{z} - \overline{z}^2\sigma_1) R + \overline{z}^2 R^2 \right] w(z) \nonumber \\
		&\quad = -z^2 w^*(z)  \left[ P_1(\overline{z}) I_2 + (\overline{z} - \overline{z}^2\sigma_1) R + \overline{z}^2 R^2 \right] B.
	\end{align}
	Recall $w(z) = v_1 (1, 0, -z^2)^\top$. By analytic continuation, this identity holds for all $z \in \mathbb{C}$. Expanding both sides as polynomials in $z$ and comparing coefficients of $z^k$ for $k = 0, 1, \dots, 4$ yields, 
		\begin{equation}\label{coeff_z0}
			A (\sigma_2 I_2 - \sigma_1 R + R^2) \mathbf{e}_1 = \mathbf{e}_3^\top (\sigma_2 I_2 - \sigma_1 R + R^2) B,
		\end{equation} 
		\begin{equation}\label{coeff_z1}
			A (-\sigma_1 I_2 + R) \mathbf{e}_1 = \mathbf{e}_3^\top (-\sigma_1 I_2 + R) B,
		\end{equation}
		\begin{equation}
			 A \mathbf{e}_1 - A(-\sigma_1 R + R^2)\mathbf{e}_3
		 - \mathbf{e}_3^\top B + \mathbf{e}_1^\top (-\sigma_1 R + R^2) B = 0, \label{coeff_z2}
		\end{equation} 
	\begin{equation}
	  A(-\sigma_1 I_2 + R)\mathbf{e}_3 -  \mathbf{e}_1^\top (-\sigma_1 I_2 + R) B = 0, \label{coeff_z3}
		\end{equation} 
and
		\begin{equation}\label{coeff_z4}
			A \mathbf{e}_3 = \mathbf{e}_1^\top B.
		\end{equation}
	Equation \eqref{coeff_z4} immediately implies $a_3 = b_1$.
	
	Next, we analyze the first equation of \eqref{system01_2}. Let $Q(z) = (z-\lambda_1)(z-\lambda_2) = z^2 - \sigma_1 z + \sigma_2$. Multiplying by $Q(z)$ yields
	\begin{align}\label{eq:1c2}
		2v_0 Q(z) &+ A \left( Q(z)I_2 + (z-\sigma_1) R + R^{2} \right) w(z) \nonumber \\
		& - z^2 v_0 A \left( Q(z)I_2 + (z-\sigma_1) R + R^{2} \right) B = 0.
	\end{align}
	Evaluating at $z=0$ (so $Q(0)=\sigma_2$) gives
	\begin{equation}\label{eq:c21z0}
		A \left( \sigma_2 I_2 - \sigma_1 R + R^{2} \right) \mathbf{e}_1 = 2v_0 \sigma_2.
	\end{equation}
 Let $P(z) = (1-z\overline{\lambda}_1)(1-z\overline{\lambda}_2) = \overline{\sigma}_2 z^2 - \overline{\sigma}_1 z + 1$. The third equation of \eqref{system01_2} can be written as,
	\begin{align}\label{eq:3c2}
		v_0 B^* & \left( P(z)I_2 + (z-z^2\overline{\sigma}_1) R^* + z^2R^{*2} \right) B^* \nonumber \\
		& + B^* \left( P(z)I_2 + (z-z^2\overline{\sigma}_1) R^* + z^2R^{*2} \right) w(z)- 2z^2 v_0 P(z) = 0.
	\end{align}
The coefficients of $z^4$ gives
	\begin{equation}\label{eq:c21z4}
	v_1	B^* \left( \overline{\sigma}_2 I_2 - \overline{\sigma}_1 R^* + R^{*2} \right) \mathbf{e}_1 = -2v_0 \overline{\sigma}_2.
	\end{equation}

	We now turn to the second vector equation in \eqref{system01_2}. Writing first the resolvent terms explicitly,
	\begin{align}
		v_0\mathbf{R}^*(z, \overline{\lambda}_1, \overline{\lambda}_2)A^* &= v_0 \left( I_2 + \overline{c_1(z)} R^* + \overline{c_2(z)} R^{*2} \right) A^*, \label{app:term1} \\
		K_z^2(R) w(z) &= 2v_1 \begin{bmatrix} 1 \\ 0 \\ -z^2 \end{bmatrix} + v_1 \left( \overline{c_1(z)} R^* + c_1(z) R \right) \begin{bmatrix} 1 \\ 0 \\ -z^2 \end{bmatrix} \nonumber \\
		&\quad + v_1 \left( \overline{c_2(z)} R^{*2} + c_2(z) R^2 \right) \begin{bmatrix} 1 \\ 0 \\ -z^2 \end{bmatrix}, \label{app:term2} \\
		-v_0 z^2 \mathbf{R}(\overline{z}, \lambda_1, \lambda_2) B &= -v_0 z^2 B - v_0 z^2 c_1(z) R B - v_0 z^2 c_2(z) R^2 B. \label{app:term3}
	\end{align}
	Multiplying the sum of \eqref{app:term1}--\eqref{app:term3} by $P(z)\overline{P}(z)$, where
	\begin{align*}
		P(z) &= (1-z\overline{\lambda}_1)(1-z\overline{\lambda}_2) = \overline{\sigma}_2 z^2 - \overline{\sigma}_1 z + 1, \\
		\overline{P}(z) &= \overline{z}^2 (z-\lambda_1)(z-\lambda_2) = \overline{z}^2 Q(z),
	\end{align*}
we obtain then the following polynomial vector equation,
	\begin{align}\label{app:vector_cleared}
		v_0 &\Bigl( Q(z) P(z)  +  zQ(z) (1 - z\overline{\sigma}_1) R^*  +  z^2Q(z) R^{*2}\Bigr)   A^* \nonumber \\
		&+ v_1 \Bigl( Q(z) P(z)  +  zQ(z) (1 - z\overline{\sigma}_1) R^*  +  z^2Q(z) R^{*2}\Bigr)  \begin{bmatrix} 1 \\ 0 \\ -z^2 \end{bmatrix} \\
		&+ v_1 \Bigl( Q(z) P(z)+ P(z) (z - \overline{\sigma}_1) R +P(z) R^2 \Bigr)  \begin{bmatrix} 1 \\ 0 \\ -z^2 \end{bmatrix} \nonumber \\
		&\qquad \qquad - v_0 z^2\Bigl(  Q(z) P(z)  -  P(z) (z - \overline{\sigma}_1) R  - P(z) R^2\Bigr)   B = 0.\nonumber 
	\end{align}
	Extracting the coefficients of $z^0$ and $z^6$ gives, respectively,
	\begin{align}
		v_0 \sigma_2 A^* + v_1 \left( R^2 - \sigma_1 R + 2\sigma_2 I_2 \right) \mathbf{e}_1 &= 0, \label{eq2case2:z0} \\
		v_1 \left( R^{*2} - \overline{\sigma}_1 R^* + 2\overline{\sigma}_2 I_2 \right) \mathbf{e}_3 + v_0 \overline{\sigma}_2 B &= 0. \label{eq2case2:z6}
	\end{align}
	The coefficients of $z^2$ and $z^4$ yield the following vector equations,
	\begin{equation}\label{eq2case2:z2}   
		\begin{aligned}
			&v_0\Bigl( (1 + |\sigma_1|^2 + |\sigma_2|^2) +  (\sigma_1 + \sigma_2 \overline{\sigma}_1) R^*  + \sigma_2 R^{*2}\Bigr) A^* \\
			&+ v_1 \Bigl(\sigma_2 R^{*2} -(\sigma_1 + \sigma_2 \overline{\sigma}_1) R^* - (\overline{\sigma}_1 + \sigma_1 \overline{\sigma}_2) R + \overline{\sigma}_2 R^2 2v_1 (1 + |\sigma_1|^2 + |\sigma_2|^2)\Bigr) \mathbf{e}_1 \\
			&-v_1 \left( R^{2} - \overline{\sigma}_1 R + 2\sigma_2 I_2 \right) \mathbf{e}_3- v_0\Bigl( \sigma_2  - \sigma_1 R  + R^2\Bigr) B = 0,
		\end{aligned}
	\end{equation}
	and \begin{align}
		&	v_0 \Bigl( R^{*2} - \overline{\sigma}_1 R^* + \overline{\sigma}_2 I_2\Bigr) A^*+ v_1\Bigl( R^{*2} - \overline{\sigma}_1 R^* + 2\overline{\sigma}_2 I_2 \Bigr) \mathbf{e}_1 \nonumber \\
		&+ v_1\Bigl( -\sigma_2 R^{*2} + (\sigma_1 + \sigma_2 \overline{\sigma}_1) R^* + (\overline{\sigma}_1 + \sigma_1 \overline{\sigma}_2) R - \overline{\sigma}_2 R^2 - 2(1 + |\sigma_1|^2 + |\sigma_2|^2)\Bigr) \mathbf{e}_3  \label{eq2case2:z4}\\
		&- v_0 \Bigl(\overline{\sigma}_2 R^2 - (\overline{\sigma}_1 + \sigma_1 \overline{\sigma}_2) R + (1 + |\sigma_1|^2 + |\sigma_2|^2) I_2\Bigr) B = 0, \nonumber 
	\end{align}
	Taking the inner product of \eqref{eq2case2:z2} with $\mathbf{e}_1$ and adding it to the adjoint of \eqref{eq2case2:z4} applied to $\mathbf{e}_3$, then using  \eqref{coeff_z1}--\eqref{coeff_z3}, \eqref{eq2case2:z0}, and \eqref{eq2case2:z6}, we obtain after straightforward calculus 
	\begin{equation}\label{key_real2}
		v_0\left(1-\left|\sigma_2 \right|^2  \right) \left(\operatorname{Re}(a_1) -\operatorname{Re}(b_3)\right)-v_1\left( \sigma_1(\overline{r}_{11}+\overline{r}_{33})-\overline{\sigma}_1(r_{11}+r_{33})\right) = 0. 
	\end{equation}
	Since the second term must be purely imaginary, it vanishes, forcing
	\begin{equation}\label{key_real_eq}
		\operatorname{Re}(a_1) = \operatorname{Re}(b_3).	
	\end{equation}
	
	Rearranging \eqref{eq2case2:z0} and \eqref{eq2case2:z6} gives
	\begin{align}
		v_1 \left( R^2 - \sigma_1 R + \sigma_2 I_2 \right) \mathbf{e}_1 &= -v_0 \sigma_2 A^* - v_1\sigma_2 \mathbf{e}_1, \label{eqcase2:z0a} \\
		v_1 \left( R^{*2} - \overline{\sigma}_1 R^* + \overline{\sigma}_2 I_2 \right) \mathbf{e}_3 &= -v_0 \overline{\sigma}_2 B - v_1\overline{\sigma}_2 \mathbf{e}_3. \label{eqcase2:z0b}
	\end{align}
Thus
\begin{align*}
	v_1 A \left( R^2 - \sigma_1 R + \sigma_2 I_2 \right) \mathbf{e}_1 &= -v_0 \sigma_2 AA^* - v_1\sigma_2 A\mathbf{e}_1, \\
	v_1 B^*\left( R^{*2} - \overline{\sigma}_1 R^* + \overline{\sigma}_2 I_2 \right) \mathbf{e}_3 &= -v_0 \overline{\sigma}_2 B^*B - v_1\overline{\sigma}_2 B^*\mathbf{e}_3. 
\end{align*}
Hence by \eqref{eq:c21z0} and \eqref{eq:c21z0}, we get
\begin{align*}
v_0 \sigma_2 AA^* + v_1\sigma_2 A\mathbf{e}_1 &= 2v_1\sigma_2, \\
v_0 \overline{\sigma}_2 B^*B + v_1\overline{\sigma}_2 B^*\mathbf{e}_3 &=2v_1\overline{\sigma}_2 . 
\end{align*}
	Assuming $\sigma_2 = \lambda_1\lambda_2 \neq 0$, taking in account \eqref{key_real_eq}, then both $a_1$ and $b_3$ are real and satisfy the same quadratic equation
	\begin{equation}\label{quadratic_general}
		v_0 x^2 + v_1 x + C = 0,
	\end{equation}
	so,
	\begin{equation}\label{sum_roots}
		a_1 + b_3 = -\frac{v_1}{v_0}.
	\end{equation}
	On the other hand, direct spectral computation using trace properties shows  $\operatorname{Tr}(R^2 - \sigma_1 R) = -2\sigma_2$. Computing the trace via the standard basis and using \eqref{eqcase2:z0a}--\eqref{eqcase2:z0b} yields
	\begin{align*}
		v_1 \operatorname{Tr}(R^2 - \sigma_1 R) &=v_1\left\langle (R^2 - \sigma_1 R)\mathbf{e}_1, \mathbf{e}_1 \right\rangle+v_1\left\langle (R^2 - \sigma_1 R)\mathbf{e}_2, \mathbf{e}_2 \right\rangle+v_1\left\langle (R^2 - \sigma_1 R)\mathbf{e}_3, \mathbf{e}_3 \right\rangle \\
		&= \left\langle -v_0 \sigma_2 A^* - v_1\sigma_2 \mathbf{e}_1, \mathbf{e}_1 \right\rangle + v_1\left\langle (R^2 - \sigma_1 R)\mathbf{e}_2, \mathbf{e}_2 \right\rangle \\
		&\quad + \left\langle \mathbf{e}_3, -v_0 \overline{\sigma}_2 B - v_1 \overline{\sigma}_2 \mathbf{e}_3 \right\rangle \\
		&= -v_0(a_1+b_3)\sigma_2 - 4v_1 \sigma_2 + v_1\left\langle (R^2 - \sigma_1 R)\mathbf{e}_2, \mathbf{e}_2 \right\rangle \\
		&= -3v_1 \sigma_2 + v_1\left\langle (R^2 - \sigma_1 R)\mathbf{e}_2, \mathbf{e}_2 \right\rangle.
	\end{align*}
	Dividing by $v_1$ and comparing with $\operatorname{Tr}(R^2 - \sigma_1 R) = -2\sigma_2$ forces
	\[
	\left\langle (R^2 - \sigma_1 R)\mathbf{e}_2, \mathbf{e}_2 \right\rangle = \sigma_2.
	\]
	Substituting this back into the trace expansion of $R^2 - \sigma_1 R - \sigma_2 I_2$ gives
\begin{align*}
		\operatorname{Tr}(R^2 - \sigma_1 R - \sigma_2 I_2) &= \left\langle (R^2 - \sigma_1 R - \sigma_2 I_2) \mathbf{e}_1, \mathbf{e}_1 \right\rangle + v_1\left\langle(R^2 - \sigma_1 R - \sigma_2 I_2)\mathbf{e}_3, \mathbf{e}_3\right\rangle \\
		&=\left\langle (R^2 - \sigma_1 R)   \mathbf{e}_1, \mathbf{e}_1 \right\rangle + \left\langle(R^2 - \sigma_1 R)\mathbf{e}_3, \mathbf{e}_3\right\rangle-2 \sigma_2 \\
		&=-\frac{v_0}{v_1}(a_1+b_3)\sigma_2 - 5\sigma_2 = -4\sigma_2.
\end{align*}
	However, the characteristic equation implies $\operatorname{Tr}(R^2 - \sigma_1 R - \sigma_2 I_2) = -3\sigma_2$. Equating these expressions yields $-4\sigma_2 = -3\sigma_2$, which forces $\sigma_2 = 0$. This contradicts the assumption $\lambda_1, \lambda_2 \neq 0$. 
	
	Thus, at least one of $\lambda_1$ or $\lambda_2$ must be zero, reducing the problem to the case treated in Proposition~\ref{Harnack_part_S_5}, where we established $\lambda_1 = \lambda_2 = 0$. Consequently, no operator $T$ with $\sigma(R)=\{0, \lambda_1, \lambda_2\}$ and $\lambda_1, \lambda_2 \neq 0$ can belong to the Harnack part of $S$. This completes the proof.
\end{proof}

Proposition \ref{Prop:case2} completes the spectral classification of the block $R$. Combined with Proposition \ref{Harnack_part_S_5}, we conclude that for any $T$ in the Harnack part of the truncated shift $S$ in $C_2(\mathbb{C}^{5})$, the block $R$ must be nilpotent (i.e., $\sigma(R)=\{0\}$).

\medskip

We now state our main theorem.

\begin{theorem}\label{Harnack_part_S5} 
	If $T\in C_2(\mathbb{C}^5)$ is in the Harnack part of $S$, then $T$ takes one of the following forms:
	\begin{equation*}
		T=\dfrac{2}{\sqrt{3}} \begin{bmatrix}
			0 & 1 & 0 & 0 & 0 \\
			0 & 0 & e^{i\theta} & 0 & 0 \\
			0 & 0 & 0 & e^{-i\theta} & 0 \\
			0 & 0 & 0 & 0 & 1 \\
			0 & 0 & 0 & 0 & 0
		\end{bmatrix},\quad \text{or} \quad 
		T=  \begin{bmatrix}
			0 & -\sqrt{3} & 0 & 0 & 0 \\
			0 & 0 & \dfrac{1}{\sqrt{2}}e^{i\theta} & 0 & 0 \\
			0 & 0 & 0 & \dfrac{1}{\sqrt{2}}e^{-i\theta} & 0 \\
			0 & 0 & 0 & 0 & -\sqrt{3} \\
			0 & 0 & 0 & 0 & 0
		\end{bmatrix},
	\end{equation*} 
	with $\theta \in \mathbb{R}$.	
\end{theorem}

\begin{remark}This result shows that even when the dimension is odd, the Harnack part of a truncated shift can be larger than the set described for dimension three, which consists of a single unitary orbit given by \eqref{S3}. In dimension five, the first form is unitarily equivalent to $S$,
	\begin{equation*}\label{S5}
		T = U_{\theta}^{*} S U_\theta \quad \text{with} \quad U_\theta = \operatorname{diag}\left(1, 1, e^{i\theta}, 1, 1 \right), \qquad \theta \in \mathbb{R}.
	\end{equation*}
	Moreover, $\|T\| = \|S\|$. This corresponds to the minimum possible norm among nilpotent operators in the Harnack part of $S$. In contrast, the second form is given by
	\begin{equation*}\label{S5c2}
		T = U_{\theta}^{*} \begin{bmatrix}
			0 & \sqrt{3} & 0 & 0 & 0 \\
			0 & 0 & \dfrac{1}{\sqrt{2}} & 0 & 0 \\
			0 & 0 & 0 & \dfrac{1}{\sqrt{2}} & 0 \\
			0 & 0 & 0 & 0 & \sqrt{3} \\
			0 & 0 & 0 & 0 & 0
		\end{bmatrix} U_\theta \quad \text{with} \quad U_\theta = \operatorname{diag}\left(1, -1, -e^{i\theta}, -1, 1 \right), \qquad \theta \in \mathbb{R}.
	\end{equation*}
	Consequently, the matrix $T$ lies in the unitary orbit of a matrix whose norm is strictly greater than that of the truncated shift. Thus Theorem  \ref{Harnack_part_S5} it answers negatively to the questions given in \cite{CassierBenharrat2020} for the case of odd dimension and $\rho=2$.
\end{remark}
\section{Conclusion}

In this paper, we have provided a complete description of the Harnack part of the $w_2$-normalized truncated shift $S$ in the class $C_2(\mathbb{C}^5)$. Our main result, Theorem \ref{Harnack_part_S5}, establishes that the structure of this equivalence class is richer than what was previously known for lower dimensions. Specifically, we showed that any operator $T$ Harnack equivalent to $S$ must assume one of two distinct forms: either unitary equivalent to $S$ (Case 1) or a specific structured matrix with different norm properties (Case 2).

These findings highlight a sharp contrast between the Harnack part in dimension $5$ and its counterparts in lower dimensions. Notably, our analysis proceeded without the additional assumption $\|T\| = \|S\|$, thereby allowing us to uncover the second family of solutions. This suggests that the geometry of Harnack parts in $C_\rho(H)$ becomes increasingly complex as the dimension grows, even for structured operators like truncated shifts.

Several questions remain open for future investigation:

\noindent\textbf{Question 1.} Does a similar dichotomy exist for truncated shifts of size $n+1$ with $n > 4$? If so, does the number of distinct forms increase with the dimension?

\noindent\textbf{Question 2.} Theorem \ref{Harnack_part_S5}, together with \cite[Theorem 3.3]{CaNaBe2024} and \cite[Theorems 3.1 and 3.3]{CaBeBel2018}, demonstrates that the Harnack equivalence class of truncated shifts exhibits extreme spectral rigidity: for dimensions up to $5$, all operators in the class must have the same spectrum as $S$. Does this property of extreme spectral rigidity hold for dimensions greater than $5$? In other words, is the nilpotency of $S$ preserved in its Harnack part?
	
\noindent\textbf{Question 3.} Extending these results to the general class $C_\rho(H)$ for $\rho \neq 2$ remains a challenging problem, as the kernel conditions become more intricate. Additionally, investigating the Harnack parts of other classes of operators, such as weighted shifts or nilpotent operators with different Jordan structures, could provide further insight into the classification of operators under Harnack equivalence.

\medskip

\noindent\textbf{Acknowledgments.} (Optional: Add funding information or thanks to colleagues here.)

\medskip

\noindent\textbf{Conflict of Interest.} The authors declare no conflict of interest.

\medskip

\noindent\textbf{Data Availability.} Data sharing not applicable to this article as no datasets were generated or analysed during the current study.


\end{document}